\documentclass[12pt,a4paper]{amsart}
\usepackage{version}
\usepackage{amsmath,amsthm,amssymb,latexsym,epsfig,graphicx,subfigure}
\usepackage[backref=page]{hyperref} 
\numberwithin{equation}{section}

\newtheorem{thm}{Theorem}[section]
\newtheorem{lm}[thm]{Lemma}

\theoremstyle{definition}

\theoremstyle{definition}
\newtheorem{rem}[thm]{Remark}
\newtheorem{rems}[thm]{Remarks}

\newcommand{\Rn}{\mathbb{R}^{n}}
\newcommand{\Pn}{\mathbb{P}^{n}}
\newcommand{\p}{\mathbb{P}}

\newcommand{\R}{\mathbb{R}}

\newcommand {\grtrsim} {\ {\raise-.5ex\hbox{$\buildrel>\over\sim$}}\ }
\newcommand{\e}{\varepsilon}

\newcommand{\khii}{\text{\lower -.4ex\hbox{$\chi$}}}
\DeclareMathOperator{\spt}{spt}

\DeclareMathOperator{\spn}{span}

\begin{document}
\title {On the parabolic Hausdorff dimension of orthogonal projections}
\author{Terence L.~J.~Harris}
\address{School of Mathematics and Physics\\ The University of Queensland\\ St Lucia QLD 4067\\  Australia}
\email{terry.harris@uq.edu.au}
\author{Pertti Mattila}
\address{Department of Mathematics and Statistics \\
P.O. Box 68 \\  FI-00014 University of Helsinki \\ Finland,}
\email{pertti.mattila@helsinki.fi}
\begin{abstract}
    For Borel sets $A\subset\R^n\times \R$ %I
    we prove lower bounds for the parabolic Hausdorff dimension of the orthogonal projections of $A$ on generic $m$-dimensional linear subspaces of $\R^n\times \R$.
\end{abstract}
\maketitle
 \section{Introduction}

Marstrand's projection theorem \cite{Mar54} tells us the following, see, for example, \cite{Mat95} or \cite{Mat15}:

\begin{thm}\label{Marstrand}
Let $A\subset \Rn$ be a Borel set.\\
(1) If $\dim_E A\leq m$, then $\dim_E P_V(A)=\dim_E A$ for almost all $V\in G(n,m)$.\\
(2) If $\dim_E A>m$, then $\mathcal L^m(P_V(A))>0$ for almost all $V\in G(n,m)$.%\\
\end{thm} 
Here $G(n,m)$ is the Grassmannian of linear $m$-dimensional subspaces of the Euclidean $n$-space $\R^n$, $P_V:\Rn\to V$ is the orthogonal projection, $\mathcal L^m$ is the Lebesgue measure on $\R^m$, and on the $m$-planes $V\in G(n,m)$, and $\dim_E$ denotes the Hausdorff dimension with respect to the Euclidean metric.

In this paper we shall investigate analogous results for the Hausdorff dimension $\dim$ with respect to the parabolic metric:
let $\|\cdot\|$ be the parabolic `norm',
$$\|(x,t)\|=\sqrt{|x|^2+|t|},$$
in $\Pn=\Rn\times\R$ and $d$ the corresponding metric $d(p,q)=\|p-q\|$. Here, and later, $|x|$ is the Euclidean norm of $x$. We have 
$\dim \Pn = n+2$ and $\dim V = m+1$ for every $V\in G(n+1,m)$, which is not contained in the horizontal plane $\{(x,0):x\in\Rn\}$.

For $n=2$ we shall prove the complete result:

\begin{thm}\label{projn=2}
Let  $m=1$ or\ $2$\ and let $A\subset \p^2$ be a Borel set. 
\begin{itemize}
\item[(1)] If $\dim A \leq m+1$, then $\dim P_V(A) \geq \dim A$ for almost all $V\in G(3,m)$.
\item[(2)] If $\dim A  > m+1$, then $\mathcal L^m(P_V(A))>0$ for almost all $V\in G(3,m)$. 
\end{itemize}
\end{thm}

In general dimensions we have

\begin{thm}\label{projplanesint}
Let  $1\leq m \leq n$ and let $A\subset \p^n$ be a Borel set. 
\begin{itemize}
\item[(1)] If $\dim A \leq 2$, then $\dim P_V(A) \geq \dim A$ for almost all $V\in G(n+1,m)$.
\item[(2)] If $\dim_E A \leq 1$, then $\dim P_V(A) =2\dim_E P_V(A) =2\dim_E A \geq \dim A$ for almost all $V\in G(n+1,m)$.
\item[(3)] If $\dim A  > m
+1$, then $\mathcal L^m(P_V(A))>0$ for almost all $V\in G(n+1,m)$. 
\end{itemize}
\end{thm}

The cases $\dim A > m+1$ or $m=1$ in both theorems are quite easy, based on Theorem \ref{Marstrand} and the comparison of the Euclidean and parabolic Hausdorff dimensions as given in Section \ref{dimcomp}. 

The proof of (1) for $m=2, 2<\dim A \leq 3,$ in Theorem \ref{projn=2} uses the Fourier transform. The initial setup is the same as the Euclidean case; we apply Plancherel's theorem to the integral of $s$-energies of pushforwards of $\alpha$-dimensional measures (where $s = \alpha-\epsilon$). One advantage of the Fourier analytic approach is that the resulting integral in frequency space can be dyadically composed and understood one frequency scale at a time (where scale refers to the parabolic size of the frequency). For each fixed frequency scale $2^j$, we can then dyadically decompose physical space $(x,y)$ according to the parabolic distance from $x$ to $y$. The critical scales turn out to be $d(x,y) \sim 2^{-k}$ for all $1 \leq k \leq j/2$, in the sense that every one of these scales may attain the sharp bound, and for each critical scale we need to further decompose dyadically according to how much the parabolic distance differs from Euclidean distance (and such a decomposition does not seem possible without first using the Fourier transform to isolate the frequency scale). Sometimes the oscillatory integrals that come up do not satisfy good bounds, but since we have a 2-dimensional family of projections (unlike in \cite{Ha23}), either the averaging effect of the projection family is enough to cancel this out, or the fractal measure of the pairs $(x,y)$ with large oscillatory integrals is small, or some combination of these two effects.  

The other cases are done by geometric arguments.

The condition $\dim A  > m+1$ is clearly sharp. We believe $\dim P_V(A) \geq \dim A$ for almost all $V\in G(n+1,m)$ should hold always when $\dim A\leq m+1$. 

%We shall prove
%\begin{thm}\label{projplanesint}
%Let  $1\leq m \leq n$ and let $A\subset \p^n$ be a Borel set. 
%\begin{itemize}
%\item[(1)] If $\dim A \leq 2$, then $\dim P_V(A) \geq \dim A$ for almost all $V\in G(n+1,m)$.
%\item[(2)] If $\dim_E A \leq 1$, then $\dim P_V(A) =2\dim_E P_V(A) =2\dim_E A \geq \dim A$ for almost all $V\in G(n+1,m)$.
%\item[(3)] If $\dim A  > m
%+1$, then $\mathcal L^m(P_V(A))>0$ for almost all $V\in G(n+1,m)$.
%\item[(4)] If $n=2$ then $\dim P_V(A) \geq \dim A$ for %almost all $V\in G(3,m)$, for all $\dim A \leq m+1$. 
%\end{itemize}
%\end{thm}

%The proof of (3) is quite easy based on Theorem \ref{Marstrand} and the comparison of the Euclidean and parabolic Hausdorff dimensions as given in Section \ref{dimcomp}. The proof of (1) and (2) requires a little more work. The proof of (4) uses the Fourier transform.

%The condition $\dim A  > m+1$ in (3) is clearly sharp. We believe $\dim P_V(A) \geq \dim A$ for almost all $V\in G(n+1,m)$ should hold always when $\dim A\leq m+1$. 

Parabolic Hausdorff measures have been used for a long time in connection of parabolic PDE's. Recently parabolic analogs of the quantitative David-Semmes rectifiability, \cite{DS93}, and its connections to harmonic analysis, PDE's and other topics has been under active investigation, see \cite{BHHLN25} and \cite{HJ25} and the references given there. Analogs of the classical qualitative Besicovitch-Federer rectifiability, \cite{Fe69}, were developed in \cite{Mat22} and \cite{MMP22}.

%Huge amount of work has been done for projections and Hausdorff and other fractal dimensions, see Falconer's recent survey \cite{Fa26}.
Much work has been done for projections and Hausdorff and other fractal dimensions, see Falconer's recent survey \cite{Fa26}. Closest to our parabolic case are projection theorems in Heisenberg groups. They were initiated in \cite{BDFMT13} and \cite{BFMT12} and continued in \cite{FH16}, \cite{FO23}, \cite{Ha23}, \cite{Ha25} and \cite{Ha24}. 
%In these papers projections onto homogeneous (invariant under dilations) subspaces were studied. In our parabolic case this is quite easy, see Section \ref{Projections onto homogeneous subspaces}.

\section{Preliminaries}\label{Preli}

The metric $d$ in $\Pn$
is defined as in the Introduction. 
For $A\subset\Pn$,\ $d(A)$ stands for the parabolic diameter of $A$.

We shall denote by $f_{\sharp}\mu$ the push-forward of a measure $\mu$  under a map  $f: f_{\sharp}\mu(A)= \mu(f^{-1}(A))$. 
For $A\subset\Pn$ let $\mathcal M(A)$ be the set of Borel measures $\mu$ with compact $\spt\mu\subset A$ and with $0<\mu(A)<\infty$.

By the notation $a\lesssim b$ we mean that $a\leq Cb$ for some constant $C$, by $a\gtrsim b$ the converse inequality, and by $a\sim b$ that both $a\lesssim b$ and $a\gtrsim b$ hold. The dependence of $C$ should be clear from the context, but sometimes we may specify it writing, for example, $\lesssim_c$. 

We denote by $\mathcal L^n$ the Lebesgue measure in the Euclidean $n$-space $\Rn, n\geq 1,$ and also in $n$-dimensional linear subspaces of higher dimensional Euclidean spaces. $\sigma^{n}$ will stand for the surface measure on the unit sphere $S^{n}\subset \R^{n+1}$ normalized to $\sigma^{n}(S^n)=1$.

Let $\dim$ stand for the parabolic Hausdorff dimension:
\begin{equation*}
\dim A=
\inf\{s>0:\forall \e>0\ \exists B_i\subset\Pn\ \text{with}\ A\subset\cup_{i=1}^{\infty}B_i\ \text{and}\ \sum_{i=1}^{\infty} d(B_i)^s<\e\}.\end{equation*}
As usual, this can also be defined in terms of Hausdorff measures, but we don't need them in this paper. 
Then, $\dim\Pn=n+2$. The Euclidean diameter and Hausdorff dimension will be denoted by $d_E$ and $\dim_E$.

We shall use the Frostman and energy characterizations of Hausdorff dimension.
For $0<s<n+2$ let $k_s$ be the kernel
$$k_s(x,t)=\|(x,t)\|^{-s}=(|x|^2+|t|)^{-s/2},\ (x,t)\in\Pn.$$
Define the growth exponent and the $s$-energy of $\mu\in\mathcal M(\Pn)$ by
\begin{equation*} %\label{mufractal} 
c_{s}(\mu) = \sup_{p \in \mathbb{P}^n, r >0} \frac{ \mu(B(p,r))}{r^{s}}, \end{equation*}
and
$$I_s(\mu)=\iint k_s(p-q)\,d\mu(p)\,d\mu(q)$$

Then for Borel sets $A\subset\Pn$,
\begin{align}\label{upperdensity}    
&\dim A=\sup\{s>0:\exists \mu\in\mathcal M(A)\ \text{such that}\ c_{s}(\mu)<\infty\}\\
\label{k_s1}&=\sup\{s>0:\exists \mu\in\mathcal M(A)\ \text{such that}\ I_s(\mu)<\infty\}\\
\label{k_s2}&=\sup\{s>0:\exists \mu\in\mathcal M(A)\ \text{such that}\ c_{s}(\mu)<\infty\ \text{and}\ I_s(\mu)<\infty\}.\end{align}

For the Euclidean metric \eqref{upperdensity} follows from, for example, \cite[Theorem 2.7]{Mat15}. The proof also works for the parabolic metric. Or one could use the general Howroyd's Frostman lemma, see \cite{Ho95} or \cite[Theorem 8.17]{Mat95}.  Then \eqref{k_s1} and \eqref{k_s2} follow from \cite[Theorem 8.12]{Mat15} and its proof.

%We shall use the energy characterization of Hausdorff dimension. For $0<s<n+2$ define the kernel $$k_s(x,t)=\|(x,t)\|^{-s}=(|x|^2+|t|)^{-s/2},\ (x,t)\in\Pn,$$
%and the $s$-energy of $\mu\in\mathcal M(\Pn)$, $$I_s(\mu)=\iint k_s(p-q)\,d\mu(p)\,d\mu(q).$$
%Then for Borel sets $A\subset\Pn$,
%\begin{equation}\label{k_s}
%\dim A=\sup\{s>0:\exists \mu\in\mathcal M(A)\ \text{such that}\ I_s(\mu)<\infty\}.\end{equation}

%Using Howroyd's Frostman lemma, see \cite{Ho95} or \cite[Theorem 8.17]{Mat95}, exactly the same arguments as in the Euclidean case can be used to verify this, see, for example, the proof of \cite[Theorem 8.9]{Mat95}.

The Euclidean $s$-energy of $\mu\in\mathcal M(\Pn)$ is
$$I_{E,s}(\mu)=\iint |x-y|^{-s}\,d\mu(x)\,d\mu(y).$$

Let $G(n+1,m)$ denote the Grassmannian manifold of the $m$-dimensional (meaning linear dimension) linear subspaces of $\Pn$ equipped with the unique rotationally invariant Borel probability measure $\gamma_{n+1,m}$. Let 
$$H_n=\{(x,t)\in\Pn:t=0\}$$
be the horizontal $n$-plane and 
$$G(H_n,m)=\{V\in G(n+1,m):V\subset H_n\}$$
be the set of horizontal $m$-planes. Then, as is easily checked, for $V\in G(n+1,m)$,\ $\dim V=m$, if $V\in G(H_n,m)$, and $\dim V=m+1$ otherwise. Identifying $H_n$ with $\Rn$ we set $\gamma_{H_n,m}=\gamma_{n,m}$. The uniqueness of $\gamma_{n,m}$ implies
\begin{equation}\label{gammaperp}
 \gamma_{n,m}(G)=\gamma_{n,n-m}(\{V\in G(n,n-m):V^{\perp}\in G\}),\ G\subset G(n,m).   
\end{equation}

Let $P_V$ denote the orthogonal projection onto a subspace $V$ in some $\R^N$. Note that the projections onto horizontal and vertical (those containing the $t$-axis) planes are Lipschitz and they do not increase the Hausdorff dimension. But this is false for all other planes $V\in G(n+1,m)$; $P_V$ is only H\"older continuous with exponent $1/2$. 

We have for $\delta>0, x\in\Rn,$ see Lemma 3.11 and Corollary 3.12 in \cite{Mat95},
\begin{equation}\label{gamma} 
\gamma_{n,m}(\{V\in G(n,m):|P_V(x)|<\delta\})\lesssim \delta^{m}|x|^{-m},
\end{equation}

and for $0<s<m$,

\begin{equation}\label{gamma3}
\int|P_V(x)|^{-s}\,d\gamma_{n,m}V \lesssim |x|^{-s}.
\end{equation}

For $m=1$, \eqref{gamma} is the same as 

\begin{equation}\label{gamma2} 
\sigma^{n-1}(\{e\in S^{n-1}:|\langle e, x\rangle |<\delta\})\lesssim \delta|x|^{-1}.
\end{equation}

The following two simple lemmas will be used in Section \ref{Projections onto general subspaces}.

\begin{lm}\label{gamma1}
  Let $1\leq k\leq n-1$ and $x\in\Rn$.
  \begin{itemize}
      \item[(1)]
If $0<t< |x|/2$, then
$$\int(|P_V(x)|+t)^{-1}\,d\gamma_{n,k}(V)\lesssim |x|^{-1}\log(|x|/t).$$  
\item[(2)] If $k\geq 2$, then
  $$\int|P_V(x)|^{-1}\,d\gamma_{n,k}(V)\lesssim |x|^{-1}.$$ 
  \end{itemize}
\end{lm}

\begin{proof} To prove (1), set $v=|x|^{-1}x$ and $a=t/|x|<1/2$. Then
    \begin{align*}
\int(|P_V(x)|+t)^{-1} \,d\gamma_{n,k}(V) =
|x|^{-1}\int(|P_V(v)|+a)^{-1} \,d\gamma_{n,k}(V).\end{align*}
 We estimate this by dividing the integration into dyadic pieces. Let $N$ be the smallest integer $j$ for which $2^{j}a>1$, so $N\sim |\log a|=\log(|x|/t)$. Then, by \eqref{gamma},
\begin{align*}
&\int(|P_V(v)|+a)^{-1} \,d\gamma_{n,k}(V)\\
&\sim a^{-1}\gamma_{n,k}(\{V:|P_V(v)|<a\})+\sum_{j=1}^N(2^ja)^{-1}\gamma_{n,k}(\{V:2^{j-1}a\leq |P_V(v)|<2^ja\})\\
 &\lesssim\sum_{j=1}^N(2^ja)^{-1}(2^ja)^k\lesssim N\sim \log(|x|/t).\\
\end{align*}
(2) follows from \eqref{gamma3}
%For (2), we get in the same manner
%$$\int|P_V(x)|^{-1} \,d\gamma_{n,k}(V) =|x|^{-1}\int|P_V(v)|^{-1} \,d\gamma_{n,k}(V),$$
%and
%\begin{align*}
%&\int|P_V(v)|^{-1} \,d\gamma_{n,k}(V) \sim 
% \sum_{j=1}^{\infty}2^{j}\gamma_{n,k}(\{V:2^{-j}< |P_V(v)|\leq 2^{1-j}\})\\
%&\lesssim \sum_{j=1}^{\infty}
%2^{j}2^{-kj}= \sum_{j=1}^{\infty}2^{(1-k)j}\sim 1.\\
%\end{align*}
\end{proof}

The next lemma holds, and is easy to prove, for all Borel sets $B\subset S^k$. But since we don't need it, we only state the following trivial form.

Write $e=(e_1,e_2)\in \R^k\times\R$ when $e\in S^k$. For $0<c<1$ let
$$S^k_c=\{e\in S^k:|e_{1}|>c, |e_2|>c\},$$
$$C_v=\{e\in S^k_c:e_1/|e_1|=v\},\ v\in S^{k-1}.$$

\begin{lm}\label{spherele}
For any Borel set $B\subset S^k_c$,
$$\sigma^k(B)\sim \int \sigma^1(B\cap C_v)\,d\sigma^{k-1}(v).$$\end{lm}

\begin{proof}
The projection $\pi, \pi(e)= e_1$, is  bilipschitz on the upper and lower halves of $S_c^k$ and $C_v$ for every $v\in S^{k-1}$. Hence $\sigma^k(B)\sim \mathcal L^k(\pi(B))$ and 
$$\int \sigma^1(B\cap C_v)\,d\sigma^{k-1}(v)\sim \int \mathcal L^1(\pi(B\cap C_v))\,d\sigma^{k-1}(v).$$
Here $\int \mathcal L^1(\pi(B\cap C_v))\,d\sigma^{k-1}(v)\sim \mathcal L^k(\pi(B))$ by integration in polar coordinates, because $\pi(B)\subset \{x\in\R^k:|x|>c\}$. The lemma follows.
\end{proof}

\section{Comparison of Euclidean and parabolic Hausdorff dimensions}\label{dimcomp}

Immediately from the definition of the parabolic metric we get for any $p\in\p^n$ with $|p|<1$ and any  $A\subset\p^n$ with $d(A)<1$,
\begin{equation}\label{distcomp}
|p|\leq\|p\|\leq 2\sqrt{|p|},\ d_E(A)\leq d(A)\ \leq 2\sqrt{d_E(A)}.
\end{equation}

We have the following rather easy comparison inequalities:

Let $A\subset\p^n$ with $d(A)<1$. Then,
\begin{equation}\label{comp1}
\dim_EA\leq \dim A\leq 2\dim_EA,\end{equation}
\begin{equation}\label{comp2} \dim_EA\leq \dim A\leq \dim_EA + 1,\end{equation}
\begin{equation}\label{comp3}  2\dim_EA-n\leq \dim A\leq \dim_EA+1.
\end{equation}

\eqref{comp1} is the best of these in the range $\dim_EA\leq 1$, \eqref{comp2} in the range $1\leq \dim_EA\leq n$, and \eqref{comp3} in the range $n\leq \dim_EA\leq n+1$.

\begin{proof}
The first item and the left inequality of the second item follow immediately from \eqref{distcomp}. The right hand inequalities in \eqref{comp2} and \eqref{comp3} follow since any Euclidean ball with radius $r<1$ can be covered with roughly $1/r$ parabolic balls with radius $r$. Finally, the left hand inequality in \eqref{comp3} follows since any parabolic ball with radius $r$ can be covered with roughly $r^{-n}$ Euclidean balls with radius $r^2$. 
\end{proof}

It is easy to find examples showing that these bounds are sharp.

If $V\in G(n+1,m)\setminus G(H_n,m)$, then the orthogonal projection from $V$ onto the vertical plane $\{(x,t):x\in V\cap H_n, t\in\R\}$ is both Euclidean and parabolic bi-Lipschitz. Hence the inequalities \eqref{comp1}, \eqref{comp2} and \eqref{comp3}, with $n$ replaced by $m-1$, are valid for subsets of $V$.

Similar, but much more difficult, comparison results were proved by Balogh, Rickly and Serra Cassano in the first Heisenberg group, \cite{BRS03}, and by Balogh, Tyson and Warhurst in general Carnot groups, \cite{BTW09}.

\section{Projections onto homogeneous linear subspaces }\label{Projections onto homogeneous subspaces}

In this section we discuss the easy case of projections onto homogeneous subspaces. They are the $m$-planes $V\in G(n+1,m)$ which are invariant under the dilations $(x,t)\mapsto (rx,r^2t), (x,t)\in\Pn, r>0$. These are exactly the horizontal planes $V\subset H_n$ and their orthogonal complements. The latter are the vertical planes; those containing the $t$-axis $\{(0,t):t\in\R\}$. 

The projections onto homogeneous linear subspaces are Lipschitz maps and do not increase the parabolic Hausdorff dimension. The vertical $m$-planes planes are given by $V_W=W+T, W\in G(H_n,m-1), T=\{(0,t):t\in\R\}$. Then $\dim V_W=m+1$. Let $e^v=(0,1)\in\Pn.$ For $p=(x,t)$ the projection is given by
$$P_{V_W}:\p^n\to V_W, P_{V_W}(p)=P_W(p)+\langle e^v, p\rangle e^v=(P_W(x),t).$$
Let 
$$\pi:\p^n\to H_n, \pi(x,t)=x,$$
be the projection onto the horizontal plane.

Let  $A\subset\p^n$. For any homogeneous $U\in G(n+1,k)$,\ $A\subset P_U(A)+ U^{\perp}$. Using this, it follows easily from the definition of the parabolic Hausdorff dimension that $\dim P_U(A) \geq \dim A - \dim U^{\perp}.$ This gives for all $W\in G(H_n,m)$,
\begin{equation}\label{projhom3}
    \dim P_W(A) \geq \dim A +m- n-2,\end{equation}
\begin{equation}\label{projhom4}\dim P_{V_W}(A) \geq \dim A + m - n,\end{equation}
\begin{equation}\label{projhom5}\dim \pi(A) \geq \dim A - 2\end{equation}
All these inequalities are sharp, but we can improve them for almost all $W$.
\begin{thm}\label{projhom1}
Let $1\leq m \leq n-1$ and let $A\subset \p^n$ be a Borel set. \\
(1) If $\dim A \leq m+2$, then $\dim P_V(A)\geq \dim A - 2$ for almost all $V\in G(H_n,m)$.\\
(2) If $\dim A > m+2$, then $\mathcal L^m(P_V(A))>0$ for almost all $V\in G(H_n,m)$.
\end{thm}

\begin{proof}
Clearly, $P_V(A) = P_V(\pi(A))$ for $V\in G(H_n,m)$. Since $\dim \pi(A) \geq \dim A - 2,$ the result follows applying the Euclidean projection theorem \ref{Marstrand} to $\pi(A)$.
\end{proof}

Examples of the form $A=B\times\R$ show that $\dim A-2$ is sharp.

\begin{thm}\label{projhom2}
Let $1\leq m \leq n-1$ and let $A\subset \p^n$ be a Borel set. 
\begin{itemize}
\item[(1)] If $\dim A \leq m$, then $\dim P_{V_W}(A) \geq \dim A$ for almost all $W\in G(H_n,m)$.
\item[(2)] If $m\leq\dim A  \leq n$, then $\dim P_{V_W}(A) \geq m$ for almost all $W\in G(H_n,m)$.
\item[(3)] If $\dim A \geq n$, then $\dim P_{V_W}(A) \geq \dim A+m-n$ for all $W\in G(H_n,m)$.
\end{itemize}
\end{thm}

\begin{proof}
To prove (1), let $0<s<\dim A \leq m$. By \eqref{k_s1} there is $\mu\in\mathcal M(A)$ with $I_s(\mu)<\infty$. 
For any $(x,t)\in \p^n$ we have
\begin{align*}
&\int\|P_{V_W}(x,t)\|^{-s}\,d\gamma_{H_n,m}(W)=\int (|P_W(x)|^2+|t|)^{-s/2}\,d\gamma_{H_n,m}(W)\\
&\leq \int \min\{(|P_W(x)|^{-s},|t|^{-s/2}\}\,d\gamma_{H_n,m}(W)\leq \min\{\int |P_W(x)|^{-s}\,d\gamma_{H_n,m}(W),|t|^{-s/2}\}\\ 
&\lesssim \min\{|x|^{-s}, |t|^{-s/2}\} \sim \|(x,t)\|^{-s},
\end{align*}
where we used \eqref{gamma3} to get the second to last estimate. This implies 
\begin{align*}
&\int I_s(P_{{V_W}\sharp}\mu)\,d\gamma_{H_n,m}(W) = \iiint\|P_{V_W}(q-p)\|^{-s}\,d\mu(p)\,d\mu(q)\,d\gamma_{H_n,m}(W)\\
&\lesssim \iint\|q-p\|^{-s}\,d\mu(p)\,d\mu(q)\ = I_s(\mu)<\infty,\end{align*}
from which (1) follows by \eqref{k_s1}.

Item (2) is trivial by (1) and item (3) is the uniform estimate \eqref{projhom4}.
\end{proof}

The above estimates are sharp:

(1) and (2): For any $A\subset H_n$, $\dim P_{V_W}(A) \leq \min\{\dim A,m\}$ for all $W\in G(H_n,m)$, so (1) and (2) are sharp.

(3): Take $A=H_n\times B, B\subset\R$, where $\R$ is equipped with the parabolic metric. Then $\dim A = n+\dim B, P_{V_W}(A)=W\times B$ and $\dim P_{V_W}(A) = m+\dim B= \dim A+m-n$, whence (3) is sharp. Taking $B$ with $\dim B=2$ and $\mathcal L^1(B)=0$, we get $A$ with $\dim A=n+2$, but $\mathcal L^1(P_{V_W}(A))=0$ for  all $W\in G(H_n,m)$.

In Heisenberg groups, and more general homogeneous groups,  projections defined via the group structure onto homogeneous subgroups are more natural than the orthogonal projections onto general linear subspaces. Consequently, they have recently been studied by several authors, see \cite{BDFMT13},  \cite{BFMT12}, \cite{FH16}, \cite{FO23}, \cite{Ha23}, \cite{Ha25} and \cite{Ha24}. In Heisenberg groups these projections are much more complicated than in the above very easy parabolic case. So the main interest of the questions studied in this paper is not about homogeneous groups but about geometric measure theory in $\R^{n+1}$ with a non-Euclidean metric, in particular, behavior of parabolic Hausdorff dimension under general orthogonal projections.

\section{Projections onto general linear subspaces}\label{Projections onto general subspaces}

Next we consider projections onto general lines and planes in $\p^n$.
Recall that for $V\in G(n+1,m)$, $\dim V=m$, if $V\subset H_n$, and $\dim V=m+1$, otherwise. For non-horizontal lines $L$ we also have $\dim B = 2\dim_EB$ for any $B\subset L$.
Here is first the easy case $m=1$ of Theorem \ref{projplanesint}:

\begin{thm}\label{projlines}
Let $A\subset \p^n$ be a Borel set. 
\begin{itemize}
\item[(1)] If $\dim A \leq 2$, then $\dim P_L(A) \geq \dim A$ for almost all $L\in G(n+1,1)$.
\item[(2)] If $\dim A  > 2$, then $\mathcal L^1(P_L(A))>0$ for almost all $L\in G(n+1,1)$.
\end{itemize}
\end{thm}

\begin{proof} 
By the Euclidean projection theorem \ref{Marstrand} and the dimension comparison \eqref{comp1}, 
we have for almost all lines $L$,
$$\dim P_L(A)=2\dim_E P_L(A)=\min\{2\dim_EA,2\}\geq \dim A,$$
which gives (1).

If $\dim A > 2$, then $\dim_EA \geq \dim A/2>1$ and again (2) follows by the Euclidean projection theorem. 
\end{proof}

The bounds in Theorem \ref{projlines} are sharp; consider subsets $A$ of lines. 

Let $m>1$. Almost every $V\in G(n+1,m)$ intersects the horizontal plane $H_n$ in an $(m-1)$-plane $W\in G(H_n,m-1)$. So, it is uniquely (up to identifying $e$ and $-e$) given as the plane spanned by $W$ and $ L_e=\{ue:u\in\R\}, e\in W^{\perp}\cap S^n$. Here and later the orthogonal complement is taken in $\R^{n+1}$, whence $W^{\perp}$ is an $(n-m+2)$-dimensional vertical plane. We then denote $V=(W,e)$ and  $P_{V}=P_{W,e}$. Let 
$$G=\{(W,e):W \in G(H_n,m-1), e\in W^{\perp}\cap S^n\}\subset G(n+1,m),$$
and let $\lambda_{n+1,m}$ be the Borel probability measure on $G$, defined by
\begin{equation}\label{lambda}
    \lambda_{n+1,m}(E)=\int\sigma^{n-m+1}(\{e\in W^{\perp}\cap S^n:(W,e)\in E\})\,d\gamma_{H_n,m-1}(W)\end{equation}
for Borel sets $E\subset G$. 

For $0<c<1$, writing $e=(e_1,e_2)\in (H_n\times \R)\cap S^n$, let 
$$G_c=\{(W,e)\in G: |e_1|>c, |e_2|>c\}.$$
Then 
\begin{equation}\label{G-H}
\gamma_{n+1,m}(G(n+1,m)\setminus\bigcup_{i=1}^{\infty}G_{1/i})=0.
\end{equation}

The measures $\lambda_{n+1,m}$ and $\gamma_{n+1,m}$ are mutually absolutely continuous:

\begin{lm}
    For Borel sets $E\subset G$,\ $\lambda_{n+1,m}(E)=0$ if and only if $\gamma_{n+1,m}(E)=0.$
\end{lm}

\begin{proof}
    Let $d_{n+1,m}$ be the standard metric on $G(n+1,m)$ given by  
    $$d_{n+1,m}(V,V')=\sup\{d_E(x,V'):x\in V\cap S^n\}.$$ 
    Then $d_{n+1,m}$ is easily seen to be locally equivalent with the  metric
    $$d'((W,e),(W',e'))=d_{n,m-1}(W,W')+d_{n+1,1}(L_e,L_{e'})$$ 
    on the open sets $G_c, 0<c<1$. The dimension of $G(n+1,m)$ is $m(n+1-m)$ and $\gamma_{n+1,m}(B(V,r))\sim r^{m(n+1-m)}$ for $V\in G(n+1,m), 0<r<1$, see, for example, \cite[3.2.28(2)]{Fe69}. Here balls are defined with the metric $d_{n+1,m}$.  For the corresponding $d'$ balls the definition of $\lambda_{n+1,m}$ gives $\lambda_{n+1,m}(B'((W,e),r))\sim  r^{m(n+1-m)}$ for $(W,e)\in G_c$ and for small $r>0$. This yields the lemma.
\end{proof}

Here is Theorem \ref{projplanesint} for $m>1$:

\begin{thm}\label{projplanes}
Let  $2\leq m \leq n$ and let $A\subset \p^n$ be a Borel set. 
\begin{itemize}
\item[(1)] If $\dim A \leq 2$, then $\dim P_V(A) \geq \dim A$ for almost all $V\in G(n+1,m)$.
\item[(2)] If $\dim_E A \leq 1$, then $\dim P_V(A) =2\dim_E P_V(A) =2\dim_E A \geq \dim A$ for almost all $V\in G(n+1,m)$.
\item[(3)] If $\dim A  > m
+1$, then $\mathcal L^m(P_V(A))>0$ for almost all $V\in G(n+1,m)$.
\end{itemize}
\end{thm}

\begin{proof} 
We may assume that $d(A)<1$. 
To prove (1) and (2), we use the parametrization of the planes with $(W,e)\in G$, that is, $W \in G(H_n,m-1), e\in W^{\perp}\cap S^n$, as above. The projection $P_{W,e}$ onto the plane spanned by $(W,e)$ is given by
$$P_{W,e}(p)=P_W(p)+\langle e,p\rangle e,\ p\in\p^n.$$
By \eqref{G-H}, given $0<c<1$, it suffices to consider $P_{W,e}$ with $(W,e)\in G_c$.
Writing again $e=(e_1,e_2), e_1\in H_n, e_2\in\R,$ we have
$$P_{W,e}(x,t)=(P_W(x)+\langle e, (x,t)\rangle e_1,\langle e, (x,t)\rangle e_2)\in H_n\times\R.$$
Therefore,
$$\|P_{W,e}(x,t)\|^2=|P_W(x)+\langle e, (x,t)\rangle e_1|^2+|\langle e, (x,t)\rangle e_2|.$$
Let $0<\delta<c/2$. Then, if $\|P_{W,e}(x,t)\|<\delta$, we get $|P_W(x)+\langle e, (x,t)\rangle e_1|<\delta$ and, since $(W,e)\in G_c$, $|\langle e, (x,t)\rangle |<\delta^2/c<\delta/2$, which implies $|P_W(p)|=|P_W(x)|<2\delta$. Thus

\begin{equation}
\label{incl}\{(W,e)\in G_c:\|P_{W,e}(p)\|<\delta\}
\subset 
\{(W,e)\in G_c:|P_W(p)|<2\delta, |\langle e, p\rangle |<\delta^2/c\}.
\end{equation}

Let $p=(x,t)\in \p^n$ with $|p|<1$ and $1<\alpha<2$. We shall show that
\begin{equation}\label{sigmaest}
\lambda_{n+1,m}(\{(W,e)\in G_c:\|P_{W,e}(p)\|<\delta\})\lesssim\delta^{\alpha}|p|^{-1}.
\end{equation}
The implicit constants are allowed to depend on $n,m,c$ and $\alpha$. By \eqref{distcomp} this gives
\begin{equation}\label{sigmaest1}
\lambda_{n+1,m}(\{(W,e)\in G_c:\|P_{W,e}(p)\|<\delta\})\lesssim\delta^{\alpha}\|p\|^{-2}.\end{equation}

Suppose first that $|p|<\delta^2.$ Then   $\delta^{\alpha}|p|^{-1}> 1$ and \eqref{sigmaest} holds.

Next suppose that $|p|>2\delta.$ Then $|P_W(p)|^2+|P_{W^{\perp}}(p)|^2=|p|^2$ and $|P_W(p)|<\delta$ imply $|P_{W^{\perp}}(p)|>|p|/2$ and we obtain by \eqref{gamma2} and \eqref{distcomp},
\begin{align*}
&\sigma^{n-m+1}(\{e\in W^{\perp}\cap S^n:|\langle e, p \rangle|< \delta^2/c\})\\
&=\sigma^{n-m+1}(\{e\in W^{\perp}\cap S^n:|\langle e, P_{W^{\perp}}(p)\rangle |<\delta^2/c\})\\
&\lesssim\delta^{2}|P_{W^{\perp}}(p)|^{-1}\sim\delta^{2}|p|^{-1},
\end{align*}
which gives \eqref{sigmaest} by \eqref{incl} and \eqref{lambda}.

Next, assume that $\delta^2\leq |p|\leq 2\delta,$ and suppose first that, in addition, $|t|\geq\delta^3.$ We have $P_{W^{\perp}}(x,t)=(P_{W^{\perp}}(x),t)\in H_n\times\R$. Then
\begin{align*}
&\lambda_{n+1,m}(\{(W,e)\in G_c:\|P_{W,e}(p)\|<\delta\})\\
&\leq \int_{\{W:|P_W(p)|<2\delta\}}\sigma^{n-m+1}(\{e\in W^{\perp}\cap S^n:|\langle e, p \rangle|< \delta^2/c\})\,d\gamma_{H_n,m-1}(W)\\
&= \int_{\{W:|P_W(p)|<2\delta\}}\sigma^{n-m+1}(\{e\in W^{\perp}\cap S^2:|\langle e, P_{W^{\perp}}(p)\rangle|< \delta^2/c\})\,d\gamma_{H_n,m-1}(W)\\
&\lesssim \int_{\{W:|P_W(p)|<2\delta\}}\delta^{2}| P_{W^{\perp}}(p)|^{-1}\,d\gamma_{H_n,m-1}(W)\\
&\leq \delta^{2}\int(|P_{W^{\perp}}(x)|+|t|)^{-1} \,d\gamma_{H_n,m-1}(W)\\
&=\delta^{2}\int_{G(n,n-m+1)}(|P_V(x)|+|t|)^{-1} \,d\gamma_{n,n-m+1}(V),
\end{align*}
where $x\in H_n = \R^n, t\in\R$.
If $|x|\leq 2|t|$, 
\begin{align*}
\delta^{2}\int_{G(n,n-m+1)}(|P_V(x)|+|t|)^{-1} \,d\gamma_{n,n-m+1}(V)
\sim \delta^{2}|t|^{-1}\sim \delta^{2}|p|^{-1},\end{align*}
and \eqref{sigmaest} holds. Suppose $2|t|< |x|$.  
Then by Lemma \ref{gamma1},
\begin{align*}
&\delta^{2}\int_{G(n,n-m+1)}(|P_V(x)|+|t|)^{-1} \,d\gamma_{n,n-m+1}(V)\\
&\lesssim \delta^{2}|x|^{-1}\log(|x|/|t|)\lesssim\delta^{2}|\log\delta||x|^{-1}\lesssim \delta^{\alpha}|p|^{-1}.\end{align*}
 Therefore, \eqref{sigmaest} follows also in this case.

Finally, assume that $\delta^2\leq |p|\leq 2\delta$ and $|t|<\delta^3.$ We first consider the case $m=n$. We have again that if $(W,e)\in G_c$ and $\|P_{W,e}(p)\|<\delta$, then $|(\langle e_1, x\rangle+e_2t)e_2|<\delta^2$, 
so $|\langle e_1,  x\rangle|\leq 2\delta^2/c$. For the unit vector $v=e_1/|e_1|\in H_n$ orthogonal to $W$ we have $|\langle v, x\rangle|=|\langle e_1,  x\rangle|/|e_1|\leq 2\delta^2/c^2$. 

For any $W\in G(H_n,n-1)$ let $v_W\in H_n\cap S^n$ be the (two-valued) unit vector orthogonal to $W$. By the uniqueness of $\gamma_{H_n,n-1}$, 
$$\gamma_{H_n,n-1}(F)=\sigma^{n-1}(\{v_W: W\in  F\}),\ F\subset G(H_n,n-1).$$
Thus
\begin{align*}
&\lambda_{n+1,n}(\{(W,e)\in G_c:\|P_{W,e}(p)\|<\delta\})\\
&=\int\sigma^{1}(\{e\in W^{\perp}\cap S^n:\|P_{W,e}(p)\|<\delta\})\,d\gamma_{H_n,n-1}(W)\\
&\leq\gamma_{H_n,n-1}(\{W\in G(H_n,n-1): |\langle v_W, x\rangle |<2\delta^2/c^2\})\\
&=\sigma^{n-1}(\{v\in H_n\cap S^n: |\langle v, x\rangle|<2\delta^2/c^2\})\\
&\lesssim \delta^{2}|x|^{-1}\sim\delta^{2}|p|^{-1},
\end{align*}
and we again have \eqref{sigmaest}.

We shall now finish the proof of \eqref{sigmaest} in the remaining case $2\leq m\leq n-1$ and $\delta^2\leq |p|\leq 2\delta, |t|<\delta^3$. For $W\in G(H_n,m-1)$ we apply Lemma \ref{spherele} in $W^{\perp}\in G(n+1,n-m+2)$ with $k=n-m+1$ and $B=B_W=\{e\in W^{\perp}\cap S^n:(W,e)\in G_c, \|P_{W,e}(p)\|<\delta\}$. As above, if $v\in W^{\perp}\cap H_n\cap S^n$, $C_v=\{e\in W^{\perp}\cap S^n:e_1/|e_1|=v\}$ and $C_v\cap B_W\neq\emptyset$, so that $(W,e)\in G_c$ and $\|P_{W,e}(p)\|<\delta$ for some $e\in C_v\cap B_W$, then $|\langle v, x\rangle|\leq 2\delta^2/c^2$. Hence, by Lemma \ref{spherele} and \eqref{gamma2},
\begin{align*}
    &\sigma^{n-m+1}(\{e\in W^{\perp}\cap S^n:(W,e)\in G_c, \|P_{W,e}(p)\|<\delta\})\\
    &=\sigma^{n-m+1}(B_W)\sim \int_{W^{\perp}\cap H_n\cap S^n} \sigma^1(C_v\cap B_W)\,d\sigma^{n-m}(v)\\
    &\leq \sigma^{n-m}(\{v\in W^{\perp}\cap H_n\cap S^{n}:|\langle v, x\rangle|\leq 2\delta^2/c^2\})\\
    &= \sigma^{n-m}(\{v\in W^{\perp}\cap H_n\cap S^{n}:|\langle v, P_{W^{\perp}(x)}\rangle|\leq 2\delta^2/c^2\})\lesssim \delta^2|P_{W^{\perp}(x)}|^{-1}.
    \end{align*}
Consequently, by \eqref{gammaperp} and Lemma \ref{gamma1},
\begin{align*}
&\lambda_{n+1,n}(\{(W,e)\in G_c:\|P_{W,e}(p)\|<\delta\})\\
&=\int\sigma^{n-m+1}(\{e\in W^{\perp}\cap S^n:(W,e)\in G_c, \|P_{W,e}(p)\|<\delta\})\,d\gamma_{H_n,m-1}(W)\\
&\lesssim \delta^2\int|P_{W^{\perp}(x)}|^{-1}\,d\gamma_{H_n,m-1}(W)=\delta^2\int|P_V(x)|^{-1}\,d\gamma_{n,n-m+1}(V) \\
&\lesssim \delta^2|x|^{-1}\sim\delta^2|p|^{-1},
\end{align*}
and \eqref{sigmaest} is verified in all cases.

Let $0<s<\dim A\leq 2$ and $s<\alpha<2$ with $2s/\alpha < \dim A$. By \eqref{k_s1} there is  $\mu\in\mathcal M(A)$ with $I_{2s/\alpha}(\mu)<\infty$. We now apply \eqref{sigmaest1} with $\delta =r^{-1/s}<c/2$ for $r>(c/2)^{-2}$:
\begin{align*}
 &\int_{G_c} \|P_{W,e}(p)\|^{-s}\,d\lambda_{n+1,m}(W,e)\\
 &\leq (c/2)^{-2}+\int_{(c/2)^{-2}}^{\infty} \lambda_{n+1,m}(\{(W,e)\in G_c:\|P_{W,e}(p)\|^{-s}>r\})\,dr\\
 &= (c/2)^{-2}+\int_{(c/2)^{-2}}^{\infty} \lambda_{n+1,m}(\{(W,e)\in G_c:\|P_{W,e}(p)\|<r^{-1/s}\})\,dr\\
 &\lesssim 1+\|p\|^{-2s/\alpha}+\int_{\|p\|^{-2s/\alpha}}^{\infty}r^{-\alpha/s}\|p\|^{-2}\,dr\\ 
 &\sim 1+\|p\|^{-2s/\alpha}+\|p\|^{-2s/\alpha}=1+2\|p\|^{-2s/\alpha}.\\
\end{align*}
Hence
\begin{align*}
    &\int_{G_c} I_s(P_{W,e\sharp}(\mu))\,d\lambda_{n+1,m}(W,e) \\
    &=\int_{G_c}\iint\|P_{W,e}(p-q)\|^{-s}\,d\mu(p)\,d\mu(q)\,d\lambda_{n+1,m}(W,e)\\
    &\lesssim \mu(A)^2+\iint\|p-q\|^{-2s/\alpha}\,d\mu(p)\,d\mu(q) = \mu(A)^2+I_{2s/\alpha}(\mu)<\infty.
    \end{align*}
It follows that for all $0<s<\dim A$,\ $\dim P_{W,e}(A) \geq s$ for almost all $(W,e)\in G_c$, whence $\dim P_{W,e}(A) \geq \dim A$ for almost all $(W,e)\in G_c$. Due to \eqref{G-H}, this proves (1). 

To prove (2), let $0<s<\dim_EA\leq 1$ and   $2s<\alpha<2$ with $2s/\alpha<\dim_EA$. Choose $\mu\in\mathcal M(A)$ with $I_{E,2s/\alpha}(\mu)<\infty$. Use then \eqref{sigmaest} and the same argument as above to obtain

$$\int_{G_c} I_{2s}(P_{W,e\sharp}(\mu))\,d\lambda_{n+1,m}(W,e) \lesssim  \mu(A)^2+I_{E,2s/\alpha}(\mu)<\infty.$$
This yields $\dim P_{W,e}(A) \geq 2\dim_E A$ for almost all $(W,e)\in G_c$. By the dimension comparison and the Euclidean projection theorem, we have $\dim P_{W,e}(A) \leq 2\dim_E P_{W,e}(A) = 2\dim_E A$ for almost all $(W,e)$. (2) follows from this.

(3) follows immediately from the corresponding Euclidean result, Theorem \ref{Marstrand}, and \eqref{comp2}: if $\dim A  > m+1$, then $\dim_E A  > m$, whence $\mathcal L^m(P_V(A))>0$ for almost all $V\in G(n+1,m)$. 
\end{proof}

\begin{rems}
Under the restrictions $\dim A\leq 2$ and $\dim_EA\leq 1$, (1) and (2) of Theorem \ref{projplanesint} are sharp. What would be the sharp inequalities in general? Probably for (1), $\dim P_V(A)\geq \dim A$ for almost all $V\in G(n,m+1)$ could hold always when $\dim A\leq m+1$. This is open when $n \geq 3$ and $m \geq 2$, and would be sharp because if $A\subset W$ for some non-horizontal $W\in G(n+1,m)$, then $\dim P_V(A)=\dim A$ for almost all $V\in G(n+1,m)$.  For (2), when $\dim_EA\leq m$, there should be an inequality $\dim P_V(A)\geq C\dim_E P_V(A)$ for almost all $V\in G(n+1,m)$ with some $C>1$. If so, what would be the best value of $C$ and would it depend only on $m$ and $n$ or also on the Euclidean and/or parabolic Hausdorff dimensions of $A$? When $\dim_EA\leq 1$, it is $2$ by Theorem \ref{projplanesint}.

\end{rems}

\section{Proof of Theorem~\ref{projn=2}}

In this section, $d_E$ will refer to Euclidean distance.

%\begin{thm} If $A \subseteq \mathbb{R}^2 \times \mathbb{R}$ is Borel with $2 < \dim A \leq 3$, %then $\dim(P_V(A)) \geq \dim A$ for a.e.~$V \in G(3,2)$. \end{thm}

\begin{proof}[Proof of Theorem~\ref{projn=2}] 
The case $\dim A \leq 2$ already follows from case (1) of Theorem~\ref{projplanesint}, so it may be assumed that $m=2$ and $2 < \dim A \leq 3$. Using Frostman's lemma \eqref{upperdensity}, let $\mu$ be a compactly supported nonzero Borel measure on $A$ with $c_{\alpha}(\mu) \leq 1$, where $2 < \alpha<  \dim A$. %, where 
%\begin{equation} \label{mufractal2} c_{\alpha}(\mu) = \sup_{x \in \mathbb{R}^3, r >0} \frac{ \mu(B(x,r))}{r^{\alpha}}, \end{equation} and the balls are with respect to the parabolic metric. 
By \eqref{k_s2}%(see \cite[Appendix B]{BP17} and \cite[Proof of Theorem~2.8]{Mat15})
, it suffices to show that for any  $2 < s < \alpha$, for any $\epsilon >0$,
\begin{equation} \label{energybound}\int_{G_{\epsilon}(3,2)} I_{s}(P_{V \sharp} \mu) \, d\gamma_{3,2}(V) < \infty, \end{equation}
where $G_{\epsilon}(3,2)$ is the subset of $V \in G(3,2)$ with unit normal $v \in S^2_{+,\epsilon}$, where 
\begin{equation} \label{S2plusepsilon} S^2_{+, \epsilon} = \left\{ v \in S^2_+ : d_E(v, (0,0,1))> \epsilon, d_E(v, \mathbb{R}^2 \times \{0\} ) > \epsilon\right\}, \end{equation}
and $S^2_+$ is the upper hemisphere of $S^2$.  By scaling it may be assumed that $\mu$ is supported in the unit ball.

 Each $V \in G_{\epsilon}(3,2)$ can be written as $V = \spn\{v_1,v_2\}$, where $v_1=v_1(V)$ is a horizontal unit vector in $\mathbb{R}^2 \times \{0\}$, and $v_2=v_2(V)$ is a unit vector orthogonal to $v_1$ which makes an angle $\gtrsim \epsilon$ with the horizontal plane $\mathbb{R}^2 \times \{0\}$. More precisely, if $V = v^{\perp}$ where $v = \left(v^{(1)},v^{(2)},v^{(3)}\right) \in S^2_{+, \epsilon}$, then $v_1$ is the unit vector obtained by scaling $\left(-v^{(2)},v^{(1)},0\right)$, and $v_2$ is the unit vector obtained by scaling $\left(-v^{(1)}v^{(3)},-v^{(2)}v^{(3)}, \left(v^{(1)}\right)^2+\left(v^{(2)}\right)^2\right)$.

For $V \in G_{\epsilon}(3,2)$, write $v_2 = v_2' + v_2''$, where $v_2'$ is in the horizontal plane and $v_2''$ is parallel to $(0,0,1)$ with $|v_2''| \gtrsim \epsilon$ (note that $v_2'$ and $v_2''$ may not be in $V$). Then $v_2'$ and $v_2''$ are orthogonal to $v_1$. Thus, if $x = x_1 v_1 + x_2 v_2 \in V$, and $y = y_1 v_1 + y_2 v_2 \in V$, then
\[ d(x,y) = \sqrt{ (x_1-y_1)^2 + (x_2-y_2)^2 |v_2'|^2 + \left\lvert (x_2-y_2)v_2''\right\rvert}. \]
But $|v_2'| \leq 1$ and $|v_2''| \gtrsim \epsilon$, so this simplifies to 
\begin{align*} d(x,y) &\sim |x_1-y_1| + |x_2-y_2|^{1/2} \\
&\sim \left( |x_1-y_1|^4 + |x_2-y_2|^2\right)^{1/4}, \end{align*}
in any bounded set around the origin, where the implicit constant depends only on $\epsilon$ and the distance from the origin. Therefore, \eqref{energybound} is equivalent to 
\begin{equation} \label{energybound2} \int_{S^2_{+, \epsilon}} \int \int \frac{1}{\left(\left\lvert \langle x-y, v_1 \rangle \right\rvert^4 + \left\lvert \langle x-y, v_2 \rangle \right\rvert^{2}\right)^{s/4} } \, d\mu(x) \, d\mu(y) \, d\sigma(v) < \infty, \end{equation} 
where $\sigma$ is the surface measure on $S^2$. For $v \in S^2_{+,\epsilon}$, let $\pi_v: \mathbb{R}^3 \to \mathbb{R}^2$ be $\pi_v(x) = \left( \langle x, v_1 \rangle, \langle x, v_2 \rangle\right)$. Then \eqref{energybound2} can be written as 
\begin{equation} \label{energybound3prime} \int_{S^2_{+, \epsilon}} \int f_s(x) \, d\left( \iota_{\sharp}\pi_{v\sharp} \mu \ast \pi_{v\sharp}\mu\right)(x) \, d\sigma(v) < \infty, \end{equation} 
where $\iota(x)=-x$, and
\[ f_s(u,t) = \left(|u|^4 +t^2\right)^{s/4}, \qquad (u,t) \in \mathbb{R}^2. \]
By Plancherel's theorem (more precisely \cite[Lemma~3.3]{Ha23}), to prove \eqref{energybound3prime} it is enough to show that 
\begin{equation} \label{energybound3} \int_{S^2_{+, \epsilon}} \int_{\mathbb{R}^2} \widehat{f_s}(\xi) \left\lvert \widehat{ \pi_{v\sharp}\mu}(\xi) \right\rvert^2 \, d\xi \, d\sigma(v) < \infty, \end{equation} 
where the left-hand side is well-defined since $\widehat{f_s}>0$ when $1 < s < 3$ (see \cite[Lemma~3.2]{Ha23}). By \cite[Lemma~3.2]{Ha23}, $\widehat{f_s} \lesssim f_{3-s}$ when $1 < s < 3$, so by dyadically decomposing the frequency space $\mathbb{R}^2$ into regions where $|\xi_1| +|\xi_2|^{1/2} \sim 2^{j}$, to prove \eqref{energybound3} it suffices to show that 
%\begin{equation} \label{energybound4}
\[ \sum_{j=1}^{\infty} 2^{-j(3-s)} \int_{S^2_{+, \epsilon}} \int_{[-2^j, 2^j] \times [-2^{2j}, 2^{2j}]}  \left\lvert \widehat{ \pi_{v\sharp}\mu}(\xi) \right\rvert^2 \, d\xi \, d\sigma(v) < \infty, %\end{equation}
\]
where boundedness of the $j = 0$ term follows from the $L^{\infty}$ bound on $\widehat{ \pi_{v\sharp} \mu}$ and the local integrability of $f_{3-s}$. It is enough to show that for $2 < s < \alpha$, for any $j \geq 1$,  
\begin{equation} \label{bigoh}  \int_{S^2_{+, \epsilon}} \int_{[-2^j, 2^j] \times [-2^{2j}, 2^{2j}]}  \left\lvert \widehat{ \pi_{v\sharp}\mu}(\xi) \right\rvert^2 \, d\xi \, d\sigma(v) \ll 2^{j(3-s)}, \end{equation}
where $A_j \ll B_j$ in \eqref{bigoh} is an abbreviation for $A_j \lesssim j^{-2} B_j$ (the use of $j^{-2}$ is not important; any summable series in $j$ with slower than exponential decay would work too). By using $|z|^2 = z\overline{z}$ and writing out the Fourier transform, \eqref{bigoh} is equivalent to 
\[ \int \int \int_{S^2_{+, \epsilon}} \int_{[-2^j, 2^j] \times [-2^{2j}, 2^{2j}]}  e^{ 2\pi i \langle \pi_v(x-y), \xi \rangle} \, d\xi \, d\sigma(v) \, d\mu(x) \, d\mu(y) \ll 2^{j(3-s)}, \]
for any $j \geq 1$. To avoid complicating the notation in what follows, $z \ll r$ will mean $|z| \ll r$ in case $z$ is complex and $r$ is real. Given $j$, dyadically decompose the domain $(x,y)$ according to the parabolic distance $d(x,y) \sim 2^{-k}$ from $x$ to $y$, where $1 \leq k \leq j$. Since $s< \alpha$, the bound for the contribution from $d(x,y) < 2^{-j}$ follows easily from the fractal condition $c_{\alpha}(\mu) \leq 1$ on $\mu$ and the trivial upper bound of $2^{3j}$ for the oscillatory integral. Therefore, it suffices to prove that 
\begin{multline} \label{required} \sum_{k=1}^j \iint_{d(x,y) \sim 2^{-k}} \int_{S^2_{+, \epsilon}} \int_{[-2^j, 2^j] \times [-2^{2j}, 2^{2j}]}  e^{ 2\pi i \langle \pi_v(x-y), \xi \rangle} \\
 d\xi \, d\sigma(v) \, d\mu(x) \, d\mu(y) \ll 2^{j(3-s)}. \end{multline}
 The proof of \eqref{required} will be broken into two cases. It will be shown that
 
%$\mathbf{Case\ 1}:\ \leq k\leq j/2:$
\subsection*{Case 1: %\ \ 
$1\leq k\leq j/2$ (low $k$)}
\begin{multline} \label{lowkbound} 
\sum_{1\leq k \leq j/2} \iint_{d(x,y) \sim 2^{-k}} \int_{S^2_{+, \epsilon}} \int_{[-2^j, 2^j] \times [-2^{2j}, 2^{2j}]}  e^{ 2\pi i \langle \pi_v(x-y), \xi \rangle} \\
 d\xi \, d\sigma(v) \, d\mu(x) \, d\mu(y) \ll 2^{j(3-s)}, \end{multline} 
and 

%$\mathbf{Case\ 2}:\ j/2< k\leq j:$ 
\subsection*{Case 2: %\ \ 
$j/2< k\leq j$ (high $k$)}
\begin{multline} \label{highkbound}  \sum_{j/2<k \leq j}\iint_{d(x,y) \sim 2^{-k}} \int_{S^2_{+, \epsilon}} \int_{[-2^j, 2^j] \times [-2^{2j}, 2^{2j}]}  e^{ 2\pi i \langle \pi_v(x-y), \xi \rangle} \\
 d\xi \, d\sigma(v) \, d\mu(x) \, d\mu(y) \ll 2^{j(3-s)}. \end{multline}
 For the sum over $j/2 < k \leq j$, partition the domain further according to the Euclidean distance $d_E(x,y) \sim 2^{-k} 2^{-\ell}$ where $1 \leq \ell \leq k$, where $d_E$ is the Euclidean distance. Then to prove \eqref{highkbound} it is enough to show that
\begin{multline*} %\label{highkboundell} 
\sum_{j/2<k \leq j} \sum_{\ell=1}^k \iint_{\substack{d(x,y) \sim 2^{-k} \\ d_E(x,y) \sim 2^{-k-\ell}}} \int_{S^2_{+, \epsilon}} \int_{[-2^j, 2^j] \times [-2^{2j}, 2^{2j}]}  e^{ 2\pi i \langle \pi_v(x-y), \xi \rangle} \\
 d\xi \, d\sigma(v) \, d\mu(x) \, d\mu(y) \ll 2^{j(3-s)}. \end{multline*}
This will be further subdivided into the two sub-cases

%$\mathbf{Case\ 2.1}:\ j/2< k\leq j, 1\leq l\leq j-k:$
\subsection*{Case 2.1: $ j/2< k\leq j$, \, $1\leq \ell \leq j-k$ (high $k$ and low $\ell$)}
\begin{multline} \label{highkboundell1} \sum_{j/2<k \leq j} \sum_{\ell=1}^{j-k} \iint_{\substack{d(x,y) \sim 2^{-k} \\ d_E(x,y) \sim 2^{-k-\ell}}} \int_{S^2_{+, \epsilon}} \int_{[-2^j, 2^j] \times [-2^{2j}, 2^{2j}]}  e^{ 2\pi i \langle \pi_v(x-y), \xi \rangle} \\
 d\xi \, d\sigma(v) \, d\mu(x) \, d\mu(y) \ll 2^{j(3-s)}, \end{multline}
and 

%$\mathbf{Case\ 2.2}:\ j/2< k\leq j, j-k\leq l\leq k:$
\subsection*{Case 2.2: $j/2< k\leq j$, \, $j-k\leq \ell \leq k$ (high $k$ and high $\ell$)}
\begin{multline} \label{highkboundell2}  \sum_{j/2<k \leq j} \sum_{\ell=j-k}^k \iint_{\substack{d(x,y) \sim 2^{-k} \\ d_E(x,y) \sim 2^{-k-\ell}}} \int_{S^2_{+, \epsilon}} \int_{[-2^j, 2^j] \times [-2^{2j}, 2^{2j}]}  e^{ 2\pi i \langle \pi_v(x-y), \xi \rangle} \\
 d\xi \, d\sigma(v) \, d\mu(x) \, d\mu(y) \ll 2^{j(3-s)}. \end{multline}
 The assumption $k > j/2$ ensures that $j-k \leq k$, so the subdivision makes sense. For the first sub-case \eqref{highkboundell1}, by writing the 2-dimensional $\xi$ integral as a product of two 1-dimensional integrals and calculating each explicitly,
\begin{multline} \label{pause} \eqref{highkboundell1} \lesssim \sum_{j/2<k \leq j} \sum_{\ell=1}^{j-k} \iint_{\substack{d(x,y) \sim 2^{-k} \\ d_E(x,y) \sim 2^{-k-\ell}}} \int_{S^2_{+, \epsilon}} \\
 \frac{1}{ \max\left\{2^{-j}, \left\lvert \langle v_1, x-y \rangle \right\rvert\right\} \max\left\{ 2^{-2j},  \left\lvert\langle v_2, x-y \rangle \right\rvert \right\}} \, d\sigma(v) \, d\mu(x) \, d\mu(y). \end{multline} 
The integral over $S^2_{+,\epsilon}$ can be partitioned into dyadic regions where $\left\lvert \langle v_1, x-y \rangle \right\rvert \sim 2^{-m_1}$ and $\left\lvert \langle v_2, x-y \rangle \right\rvert \sim 2^{-m_2}$, where $k+\ell \leq m_1 \leq j$ and $k+\ell \leq m_2 \leq 2j$, with $\sim$ replaced by $\lesssim$ when either $m_1=j$ or $m_2 = 2j$. It will be shown that each such dyadic region has surface area $\lesssim 2^{-m_1-m_2 +2(\ell+k)}$.

 If $\ell < k-C$ for a large constant $C$, and $u=(x-y)/|x-y|$, then $u$ makes an angle $< \epsilon^{100}$ with the horizontal plane (since otherwise it would hold
\[ 2^{-k-\ell} \sim d_E(x,y) \sim |x_3-y_3| \sim d(x,y)^2 \sim 2^{-2k}, \]
which is a contradiction if $C$ is sufficiently large depending on $\epsilon$). Therefore, the function $f_u: B_2\left(0,1-\frac{\epsilon^2}{100}\right) \to \mathbb{R}^2$ given by
\[ f_u(w_1,w_2) = \left(-w_2u_1  + w_1 u_2, -w_1u_1-w_2u_2+ \frac{u_3(w_1^2+w_2^2)}{\sqrt{1-w_1^2-w_2^2}}\right) \]
 has\footnote{This part of the argument uses $d_E(v,\mathbb{R}^2 \times \{0\} > \epsilon)$ from the definition \eqref{S2plusepsilon} of $S^2_{+,\epsilon}$}
\begin{align} \notag \det(Df_u(w_1,w_2)) &= \det\begin{pmatrix} u_2 & -u_1 \\ -u_1 + O\left(u_3\epsilon^{-10}\right) & -u_2+ O\left(u_3\epsilon^{-10}\right) \end{pmatrix} \\
\label{detbound} & = -(u_1^2+u_2^2) + O\left(u_3 \epsilon^{-10}\sqrt{u_1^2+u_2^2}\right). \end{align} 
The reason for introducing this function is that $f_u(w_1,w_2)$ is (up to a harmless scaling) the same as $\left( \langle v_1, u \rangle, \langle v_2, u \rangle \right)$ when $v=\left(w_1,w_2,\sqrt{1-w_1^2-w_2^2}\right)$.  By the inverse function theorem and \eqref{detbound}, it follows that for any unit vector $u$ with $|u_3| \lesssim \epsilon^{100}\sqrt{u_1^2+u_2^2}$, for any open set $O \subseteq \mathbb{R}^2$, 
\[ \mathcal{H}^2_E\left(f_u^{-1}(O)\right) \lesssim \mathcal{H}^2_E(O), \]
where $\mathcal{H}^2_E$ is the 2-dimensional Lebesgue measure. Applying the above with $O$ a rectangle of dimensions $\sim 2^{-m_1+\ell+k} \times  2^{-m_2+\ell+k}$ verifies the bound of $2^{-m_1-m_2 +2(\ell+k)}$ for the surface area, in the case where $\ell < k-C$ for a large constant $C$. 

The other case for the surface area bound is $k-C \leq \ell \leq k$ (and still $\ell \leq j-k \leq j/2$). In this case $2k-C \leq m_1 \leq 2k$, and hence $2^{-m_1-m_2 +2(\ell+k)} \sim 2^{-m_2+\ell+k}$, so it is enough to show that the region of $v$ where $\left\lvert \langle v_2, x-y \rangle \right\rvert \sim 2^{-m_2}$ (with $\sim$ replaced by $\lesssim$ when $m_2 = 2j$) has surface area $\lesssim 2^{-m_2 +\ell+k}$. If $F: S^2_{+,0} \to S^2_{+,0}$ is 
\begin{align*} F(w) &= \frac{1}{\left\lvert\left(-w_1w_3, -w_2w_3, w_1^2+w_2^2\right) \right\rvert }\left(-w_1w_3, -w_2w_3, w_1^2+w_2^2\right) \\
&= \left( \frac{-w_1w_3}{\sqrt{w_1^2+w_2^2}}, \frac{-w_2w_3}{\sqrt{w_1^2+w_2^2}}, \sqrt{w_1^2+w_2^2} \right), \end{align*}
then writing  $(x,y,z)=F(w)$ gives $w_3 = \sqrt{1-z^2}$, $w_1 = \frac{-xz}{\sqrt{1-z^2}}$, and $w_2=\frac{-yz}{\sqrt{1-z^2}}$. Therefore, $F$ is a self-inverse diffeomorphism of $S^2_{+,0}$ onto itself, and in particular $F$ is bi-Lipschitz onto its image when restricted to $S^2_{+,\epsilon}$. The region $\{ v \in S^2: \left\lvert \langle v, x-y \rangle \right\rvert \sim 2^{-m_2} \}$ is the intersection of the $2^{-m_2+\ell+k}$-neighbourhood of a plane through the origin with $S^2$, which has area $\sim 2^{-m_2+\ell+k}$, so the inverse image of this set under $F:  S^2_{+,\epsilon} \to S^2_+$ also has surface area $\lesssim 2^{-m_2+\ell+k}$. 

This verifies the claim that the dyadic regions described above have surface area $\lesssim 2^{-m_1-m_2 +2(\ell+k)}$, for each pair $(m_1, m_2)$. Since there are $\lesssim \log\left(2^j\right)^2$ many such pairs $(m_1,m_2)$, summing over $m_1$ and $m_2$ in \eqref{pause} yields the bound 
\[ \eqref{highkboundell1} \lesssim \log\left(2^j\right)^2 \sum_{j/2<k\leq j}  \sum_{\ell=1}^{j-k} 2^{2(\ell+k)} \iint_{\substack{d(x,y) \sim 2^{-k} \\ d_E(x,y) \sim 2^{-k-\ell}}} \, d\mu(x) \, d\mu(y). \]
For fixed $y$, the region of integration in $x$ is contained in a Euclidean cylinder of dimensions $2^{-(k+\ell)} \times 2^{-2k}$, which can be covered by $\lesssim 2^{2\ell}$ many parabolic balls of radius $2^{-(k + \ell)}$. Therefore, each such Euclidean cylinder contributes $\lesssim 2^{2\ell -\alpha(k+\ell)}$ in $\mu$-measure to the above. Hence, the above becomes
\[ \eqref{highkboundell1} \lesssim \log\left(2^j\right)^2 \sum_{j/2<k \leq j}  \sum_{\ell=1}^{j-k} 2^{2(\ell+k)} 2^{2\ell -\alpha(k+\ell)}. \]
This simplifies to 
\[ \eqref{highkboundell1} \lesssim \log\left(2^j\right)^2 \sum_{j/2<k \leq j}  \sum_{\ell=1}^{j-k} 2^{\ell(4-\alpha) -k(\alpha-2)}. \]
Since $\alpha < 3 < 4$, summing the geometric series in $\ell$ gives 
\[ \eqref{highkboundell1} \lesssim \log\left(2^j\right)^2 \sum_{j/2<k \leq j}  2^{(j-k)(4-\alpha) -k(\alpha-2)}. \]
This simplifies to 
\[ \eqref{highkboundell1} \lesssim \log\left(2^j\right)^2 \sum_{j/2<k \leq j}  2^{-2k +j(4-\alpha)}. \]
Summing the geometric series in $k$ gives 
\[ \eqref{highkboundell1} \lesssim \log\left(2^j\right)^2  2^{j(3-\alpha)}. \]
This implies \eqref{highkboundell1} in this sub-case, since $s < \alpha$. This verifies the first sub-case \eqref{highkboundell1} (high $k$ and low $\ell$). 

For the second sub-case \eqref{highkboundell2} (high $k$ and high $\ell$), by calculating the integrals explicitly as before,
\begin{multline} \label{pause2} \eqref{highkboundell2} \lesssim \sum_{j/2<k \leq j} \sum_{\ell=j-k}^{k} \iint_{\substack{d(x,y) \sim 2^{-k} \\ d_E(x,y) \sim 2^{-k-\ell}}} \int_{S^2_{+, \epsilon}} \\
 \frac{1}{ 2^{-j} \max\left\{ 2^{-2j},  \left\lvert\langle v_2, x-y \rangle \right\rvert \right\}} \, d\sigma(v) \, d\mu(x) \, d\mu(y). \end{multline}
The integral over $S^2_+$ can be subdivided into dyadic regions where $\left\lvert \langle v_2, x-y \rangle \right\rvert \sim 2^{-m}$, where $k+\ell \leq m \leq 2j$, with $\sim$ replaced by $\lesssim$ when $m = 2j$. Each such dyadic region has surface area $\sim 2^{-m +\ell+k}$ by the same reasons as above; it is the inverse image (under a bi-Lipschitz function from $S^2_{+,\epsilon}$ onto its image  in $S^2$) of the $2^{-m+\ell+k}$-neigbourhood of a plane through the origin intersected with $S^2$.

Since there are $\lesssim \log\left(2^j\right)$ many  values of $m$, summing \eqref{pause2} over $m$ using the above surface area bound yields 
\[ \eqref{highkboundell2} \lesssim \log\left(2^j\right) \sum_{j/2<k \leq j}  \sum_{\ell= j-k}^k 2^{\ell+k+j} \iint_{\substack{d(x,y) \sim 2^{-k} \\ d_E(x,y) \sim 2^{-k-\ell}}} \, d\mu(x) \, d\mu(y). \] 
As before, the Euclidean cylinder contributes $\mu$-measure $\lesssim 2^{2\ell -\alpha(k+\ell)}$, so the above becomes 
\[ \eqref{highkboundell2} \lesssim \log\left(2^j\right) \sum_{j/2<k \leq j}  \sum_{\ell=j-k}^k 2^{\ell+k+j} 2^{2\ell -\alpha(k+\ell)}. \] 
This simplifies to 
\[ \eqref{highkboundell2} \lesssim \log\left(2^j\right) \sum_{j/2<k \leq j}  \sum_{ \ell=j-k}^k 2^{\ell(3-\alpha)-k(\alpha-1)+j}. \] 
Since $\alpha < 3$, summing the geometric series over $\ell$ gives 
\[ \eqref{highkboundell2} \lesssim \log\left(2^j\right) \sum_{j/2<k \leq j}  2^{k(3-\alpha)-k(\alpha-1)+j}. \] 
This simplifies to 
\[ \eqref{highkboundell2} \lesssim \log\left(2^j\right) \sum_{j/2<k \leq j}  2^{-2k(\alpha-2)+j}. \] 
Since $\alpha >2$, summing the geometric series over $k$ gives
\[ \eqref{highkboundell2} \lesssim \log\left(2^j\right)  2^{-j(\alpha-2)+j} = \log\left(2^j\right)  2^{j(3-\alpha)}. \] 
This implies \eqref{highkboundell2} in this sub-case, since $s < \alpha$. This is the remaining sub-case for high $k$, so this verifies the case $j/2 < k \leq j$ in \eqref{highkbound}.  

It remains to prove \eqref{lowkbound} (low $k$).  As before, partition the domain according to $d_E(x,y) \sim 2^{-k} 2^{-\ell}$ where $1 \leq \ell \leq k$, where $d_E$ is the Euclidean distance. To prove \eqref{lowkbound} it suffices to show that
\begin{multline*}  \sum_{1\leq k \leq j/2} \sum_{\ell=1}^k \iint_{\substack{d(x,y) \sim 2^{-k} \\ d_E(x,y) \sim 2^{-k-\ell}}} \int_{S^2_{+, \epsilon}} \int_{[-2^j, 2^j] \times [-2^{2j}, 2^{2j}]}  e^{ 2\pi i \langle \pi_v(x-y), \xi \rangle} \\
 d\xi \, d\sigma(v) \, d\mu(x) \, d\mu(y) \ll 2^{j(3-s)}. \end{multline*}
As before, let $C$ be a constant chosen sufficiently large depending on $\epsilon$. The above will be broken into the two sub-cases

%$\mathbf{Case\ 1.1}:\ 1 \leq k\leq j/2, 1\leq l\leq k-C$: 
\subsection*{Case~1.1:  $1 \leq k\leq j/2$ \, $1\leq \ell\leq k-C$ (low $k$ and low $\ell$) } 
\begin{multline} \label{subcaseeasy}  \sum_{1\leq k \leq j/2} \sum_{\ell=1}^{k-C} \iint_{\substack{d(x,y) \sim 2^{-k} \\ d_E(x,y) \sim 2^{-k-\ell}}} \int_{S^2_{+, \epsilon}} \int_{[-2^j, 2^j] \times [-2^{2j}, 2^{2j}]}  e^{ 2\pi i \langle \pi_v(x-y), \xi \rangle} \\
 d\xi \, d\sigma(v) \, d\mu(x) \, d\mu(y) \ll 2^{j(3-s)}. \end{multline}
and

%$\mathbf{Case\ 1.2}: 1\leq k\leq j/2, k-C\leq l \leq k$:
\subsection*{Case~1.2:  $1 \leq k\leq j/2$ \, $ k-C\leq \ell \leq k$ (low $k$ and high $\ell$) }
\begin{multline} \label{subcasehard}  \sum_{1\leq k \leq j/2} \sum_{\ell=k-C}^k \iint_{\substack{d(x,y) \sim 2^{-k} \\ d_E(x,y) \sim 2^{-k-\ell}}} \int_{S^2_{+, \epsilon}} \int_{[-2^j, 2^j] \times [-2^{2j}, 2^{2j}]}  e^{ 2\pi i \langle \pi_v(x-y), \xi \rangle} \\
 d\xi \, d\sigma(v) \, d\mu(x) \, d\mu(y) \ll 2^{j(3-s)}. \end{multline}
For the first sub-case where $\eqref{lowkbound}\lesssim  \eqref{subcaseeasy}$, similarly to before, by calculating the innermost integral, 
\begin{multline*} \eqref{subcaseeasy} \lesssim \sum_{1\leq k \leq j/2} \sum_{\ell=1}^{k-C} \iint_{\substack{d(x,y) \sim 2^{-k} \\ d_E(x,y) \sim 2^{-k-\ell}}} \int_{S^2_{+, \epsilon}} \\
 \frac{1}{ \max\left\{2^{-j}, \left\lvert \langle v_1, x-y \rangle \right\rvert\right\} \max\left\{ 2^{-2j},  \left\lvert\langle v_2, x-y \rangle \right\rvert \right\}} \,  d\sigma(v) \, d\mu(x) \, d\mu(y). \end{multline*} 
Similarly to the low $\ell$ sub-case of the high $k$ case with $\ell \leq k-C$, this becomes 
\[ \eqref{subcaseeasy} \lesssim \log\left(2^j\right)^2 \sum_{1\leq k \leq j/2} \sum_{\ell=1}^{k-C} 2^{2(k + \ell)} \iint_{\substack{d(x,y) \sim 2^{-k} \\ d_E(x,y) \sim 2^{-k-\ell}}}  \, d\mu(x) \, d\mu(y). \]
By using the same bound of $2^{2\ell - \alpha(k+\ell)}$ on the $\mu$-measure of the Euclidean cylinder as before, this becomes 
\[ \eqref{subcaseeasy} \lesssim \log\left(2^j\right)^2 \sum_{1\leq k \leq j/2} \sum_{\ell=1}^{k-C} 2^{2(k + \ell)} 2^{2\ell - \alpha(k+\ell)}. \]
This simplifies to 
\[ \eqref{subcaseeasy} \lesssim \log\left(2^j\right)^2 \sum_{1\leq k \leq j/2} \sum_{\ell=1}^{k-C} 2^{\ell(4-\alpha) -  k(\alpha-2)}. \]
Since $\alpha < 3 < 4$, summing the geometric series in $\ell$ gives 
\[ \eqref{subcaseeasy} \lesssim \log\left(2^j\right)^2 \sum_{1\leq k \leq j/2}  2^{2k(3-\alpha)}. \]
Since $\alpha < 3$, summing the geometric series in $k$ gives 
\[ \eqref{subcaseeasy} \lesssim \log\left(2^j\right)^2   2^{j(3-\alpha)}. \]
This implies \eqref{subcaseeasy} in this sub-case, since $s < \alpha$. This verifies the case where $\eqref{lowkbound} \lesssim \eqref{subcaseeasy}$.

It remains to prove \eqref{subcasehard}. %This case can be written more simply as 
By allowing the (slightly ambiguous) notation $\sim$ to include dependence on $C$, this case can be abbreviated to
\begin{multline*}  \sum_{1\leq k \leq j/2}  \iint_{\substack{d(x,y) \sim 2^{-k} \\ d_E(x,y) \sim 2^{-2k}}} \int_{S^2_{+, \epsilon}} \int_{[-2^j, 2^j] \times [-2^{2j}, 2^{2j}]}  e^{ 2\pi i \langle \pi_v(x-y), \xi \rangle} \\
 d\xi \, d\sigma(v) \, d\mu(x) \, d\mu(y) \ll 2^{j(3-s)}. \end{multline*}
Write $x = (x',x'') \in \mathbb{R}^2 \times \mathbb{R}$, and dyadically partition the domain further into sets where either $|x'-y'| \lesssim 2^{-j}$, or $|x'-y'| \sim 2^{-2k- \ell}$ where $0 \leq \ell \leq j-2k$. It suffices to show that
\begin{multline} \label{edgecase}  \sum_{1\leq k \leq j/2} \iint_{\substack{d(x,y) \sim 2^{-k} \\ |x'-y'| \leq 2^{-j}}} \int_{S^2_{+, \epsilon}} \int_{[-2^j, 2^j] \times [-2^{2j}, 2^{2j}]}  e^{ 2\pi i \langle \pi_v(x-y), \xi \rangle} \\
 d\xi \, d\sigma(v) \, d\mu(x) \, d\mu(y) \ll 2^{j(3-s)},  \end{multline} 
and
\begin{multline} \label{nonedge} \sum_{1\leq k \leq j/2} \sum_{\ell=0}^{j-2k} \iint_{\substack{d(x,y) \sim 2^{-k} \\ |x'-y'| \sim 2^{-2k-\ell}}} \int_{S^2_{+, \epsilon}} \int_{[-2^j, 2^j] \times [-2^{2j}, 2^{2j}]}  e^{ 2\pi i \langle \pi_v(x-y), \xi \rangle} \\
 d\xi \, d\sigma(v) \, d\mu(x) \, d\mu(y) \ll 2^{j(3-s)}.  \end{multline} 
The edge case \eqref{edgecase} satisfies
\[ \eqref{edgecase} \lesssim  \sum_{1\leq k \leq j/2} \iint_{\substack{d(x,y) \sim 2^{-k} \\ |x'-y'| \leq 2^{-j}}} \int_{S^2_{+, \epsilon}} \frac{1}{2^{-j} \max\left\{2^{-2j}, \left\lvert \langle v_2, x-y \rangle \right\rvert\right\}} \, d\sigma(v) \, d\mu(x) \, d\mu(y). \]
Similarly to before, this becomes
\[ \eqref{edgecase} \lesssim \log\left(2^j\right) \sum_{1\leq k \leq j/2} 2^{2k+j} \iint_{\substack{d(x,y) \sim 2^{-k} \\ |x'-y'| \leq 2^{-j}}} \, d\mu(x) \, d\mu(y).  \]
By covering the $2^{-j} \times 2^{-2k}$ Euclidean cylinder with $\lesssim 2^{2j-2k}$ many parabolic balls of radius $2^{-j}$ similarly to before, this becomes
\[ \eqref{edgecase} \lesssim  \log\left(2^j\right) \sum_{1\leq k \leq j/2} 2^{2k+j} 2^{2j-2k} 2^{-\alpha j}.  \]
This simplifies to 
\[ \eqref{edgecase} \lesssim \log\left(2^j\right)^2 2^{j(3-\alpha)},  \]
which verifies \eqref{edgecase} since $\alpha >s$. 

It remains to consider the non-edge case \eqref{nonedge}. Similarly to previous cases,
\begin{multline} \label{pause3} \eqref{nonedge} \lesssim \sum_{1\leq k \leq j/2} \sum_{\ell=0}^{j-2k} \iint_{\substack{d(x,y) \sim 2^{-k} \\ |x'-y'| \sim 2^{-2k-\ell}}} \int_{S^2_{+, \epsilon}} \\
\frac{1}{\max\left\{ 2^{-j}, \left\lvert \left\langle x-y,v_1 \right\rangle\right\rvert \right\} \max\left\{ 2^{-2j}, \left\lvert \left\langle x-y,v_2 \right\rangle\right\rvert \right\}} \, d\sigma(v) \, d\mu(x) \, d\mu(y).  \end{multline} 
Dyadically partition $S^2_{+, \epsilon}$ according to $\left\lvert \langle x-y, v_1 \rangle\right\rvert \sim 2^{-m_1}$ and $\left\lvert \langle x-y, v_2 \rangle\right\rvert \sim 2^{-m_2}$, where $2k+\ell \leq m_1 \leq j$ and $2k \leq m_2 \leq 2j$, with $\sim$ replaced by $\lesssim$ when either $m_1 =j$ or $m_2 = 2j$. It will be shown that each such region has surface area $\lesssim 2^{-m_1-m_2 +2k+\ell + 2k}$.  If $m_1 \leq 2k+\ell + C$ then this follows easily from previous arguments (see the paragraph after \eqref{pause2}; using $d_E(x,y) \gtrsim 2^{-2k}$ gives the surface measure bound of $2^{-m_2+2k}$), so it may be assumed that $m_1 \gg 2k+\ell$. Similarly, it may be assumed that $m_2 \gg 2k$ (if $m_2 \leq 2k+C$, the region is contained in the $\sim 2^{2k+\ell-m_1}$-neighbourhood of a vertical plane). If $v$ is written in (one version of) spherical coordinates 
\[ v = \left( \cos \phi \cos \theta, \cos \phi \sin \theta, \sin \phi \right), \qquad 0 \leq \theta < 2\pi, \quad c\epsilon \leq \phi \leq \frac{\pi}{2} - c\epsilon, \]
which includes $S^2_{+,\epsilon}$ if $c>0$ is sufficiently small, then 
\[ v_1 = \left( - \sin \theta,  \cos \theta,0\right), \]
and 
\[ v_2 =  \left( -\sin \phi \cos \theta,-\sin \phi   \sin \theta, \cos\phi \right). \]
Since $\partial_{\phi} v_2=-v$, it must hold that $\left\lvert\langle u, \partial_{\phi} v_2 \rangle\right\rvert \sim 1$ whenever $\left\lvert \langle x-y, v_1 \rangle\right\rvert \sim 2^{-m_1}$ and $\left\lvert \langle x-y, v_2 \rangle\right\rvert \sim 2^{-m_2}$, where $u = (x-y)/|x-y|$ (this is where $m_1 \gg 2k+\ell \geq 2k$ and $m_2 \gg 2k$ get used). Therefore, the region of $v \in S^2_{+,\epsilon}$ where $\left\lvert \langle x-y, v_1 \rangle\right\rvert \sim 2^{-m_1}$ and $\left\lvert \langle x-y, v_2 \rangle\right\rvert \sim 2^{-m_2}$ has $\theta$ contained in a set of length $\lesssim 2^{2k+\ell-m_1}$, and for each $\theta$, $\phi$ is contained in a set of length $\lesssim 2^{2k-m_2}$. This verifies the claimed surface area bound of $2^{2k+\ell-m_1+2k - m_2}$. 

Substituting this surface area bound into \eqref{pause3} gives 
\begin{equation} \label{pause4}  \eqref{nonedge} \lesssim \log\left(2^j\right)^2\sum_{1\leq k \leq j/2} \sum_{\ell=0}^{j-2k} 2^{4k+\ell} \iint_{\substack{d(x,y) \sim 2^{-k} \\ |x'-y'| \sim 2^{-2k-\ell}}} \, d\mu(x) \, d\mu(y). \end{equation}
For each $y$, the region of integration in $x$ is contained in a Euclidean cylinder of dimensions $2^{-2k-\ell} \times 2^{-2k}$, which can be covered by $\lesssim 2^{2k+2\ell}$ many parabolic balls of radius $2^{-2k-\ell}$, and therefore contributes $\mu$-measure at most $2^{2k+2\ell} 2^{-\alpha(2k+\ell)}$. Substituting into \eqref{pause4} gives 
\[  \eqref{nonedge} \lesssim \log\left(2^j\right)^2\sum_{1\leq k \leq j/2} \sum_{\ell=0}^{j-2k} 2^{4k+\ell} 2^{2k+2\ell} 2^{-\alpha(2k+\ell)}. \]
This simplifies to 
\[  \eqref{nonedge} \lesssim \log\left(2^j\right)^2\sum_{1\leq k \leq j/2} \sum_{\ell=0}^{j-2k} 2^{2k(3-\alpha)+\ell(3-\alpha)}. \]
Since $\alpha < 3$, summing the geometric series in $\ell$ gives 
\[  \eqref{nonedge} \lesssim \log\left(2^j\right)^2\sum_{1\leq k \leq j/2}  2^{2k(3-\alpha)+(j-2k)(3-\alpha)}. \]
The summand is independent of $k$, so this simplifies to 
\[  \eqref{nonedge} \lesssim \log\left(2^j\right)^3 2^{j(3-\alpha)}. \]
Since $s< \alpha$, this finishes the remaining sub-case for the remaining case of low $k$ in \eqref{lowkbound}, and finishes the proof.  \end{proof}

\begin{rem} If $f_{n,s} = \left( |x|^4 + t^2 \right)^{-s/4}$ where $(x,t) \in \mathbb{R}^n \times \mathbb{R}$ and $n \geq 1$, then $\widehat{f_{n,s}}$ is given by a function which is smooth away from the origin and which satisfies 
\begin{equation} \label{gressmanstein} \left\lvert \widehat{f_{n,s}} \right\rvert \lesssim f_{n,n+2-s}, \qquad 0 < s < n+2. \end{equation}
This follows from an observation in the introduction to \cite{GS06}, that the Fourier transform of a regular homogeneous distribution is a regular homogeneous distribution. This observation was overlooked in \cite[Lemma~2.2]{Ha23}, where it was shown using modified Bessel functions that 
\[ 0 < \widehat{f_{1,s}} \lesssim f_{1,3-s}, \qquad 1 < s < 3. \]
Due to the smoothness and homogeneity mentioned above, the positivity conclusion automatically upgrades to 
\[  \widehat{f_{1,s}} \sim f_{1,3-s}, \qquad 1 < s < 3. \]
Although a proof of positivity of $\widehat{f_{n,s}}$ may require modified Bessel functions, the positivity and lower bound $\widehat{f_{1,s}} \gtrsim f_{1,3-s}$ was not crucial to the argument of this section, so \eqref{gressmanstein} may be a first step towards generalising the argument of this section to higher dimensions.  \end{rem}

%{\bf Acknowledgement}
\subsection*{Acknowledgements} We would like to thank the referee for careful refereeing.

%\vspace{1cm}
%\begin{footnotesize}
%{\sc Department of Mathematics and Statistics,
%P.O. Box 68,  FI-00014 University of Helsinki, Finland,}\\
%\emph{E-mail address:} 
%\verb"pertti.mattila@helsinki.fi" 
%\end{footnotesize}

\end{document}